\documentclass[bezier,11pt]{article}

\usepackage{amssymb}
\usepackage{latexsym}
\usepackage{tikz,epsf}

\usepackage{amsfonts}
\usepackage{amscd}
\usepackage{amsthm}
\usepackage{amsmath, amsthm}
\usepackage{graphicx}
\usepackage{color}
\usepackage{caption,subcaption}

\newtheorem{theorem}{Theorem}[section]
\newtheorem{definition}[theorem]{Definition}

\newtheorem{lemma}[theorem]{Lemma}

\newtheorem{corollary}[theorem]{Corollary}

\newcommand{\vertex}{\node[vertex]}
\tikzstyle{vertex}=[circle, draw, inner sep=0pt, minimum size=6pt]

\newcommand{\QEDmark}{\mbox{\textsc{qed}}}
\newcommand{\proofStarter}[1]{\textsc{#1} }

\begin{document}
\title{On zero-sum $\mathbb{Z}_{2j}^k$-magic graphs}
\author{J. P. Georges\\Trinity College\\Hartford, CT USA 06013\\
\texttt{john.georges@trincoll.edu}
\and
D. Mauro\\
Trinity College\\
Hartford, CT USA 06013\\
\texttt{david.mauro@trincoll.edu}
\and
K. Wash \\
Trinity College\\
Hartford, CT USA 06013\\
\texttt{kirsti.wash@trincoll.edu}
}
\date{\today}
\maketitle

\begin{abstract}
Let $G = (V,E)$ be a finite graph and let $(\mathbb{A},+)$ be  an abelian group with identity 0.    
Then  $G$ is \textit{$\mathbb{A}$-magic} if and only if there exists  a function $\phi$ from $E$ into $\mathbb{A} - \{0\}$  such that for some $c \in \mathbb{A}$, $\sum_{e \in E(v)} \phi(e) = c$ for every $v \in V$, where $E(v)$ is the set of edges incident to $v$.  Additionally, $G$ is \textit{zero-sum $\mathbb{A}$-magic} if and only if $\phi$ exists such that $c = 0$.  We consider  zero-sum $\mathbb{A}$-magic labelings of graphs,  with particular attention given to $\mathbb{A} = \mathbb{Z}_{2j}^k$. For $j \geq 1$, let $\zeta_{2j}(G)$ be the smallest positive integer $c$ such that $G$ is zero-sum $\mathbb{Z}_{2j}^c$-magic if $c$ exists; infinity otherwise. We establish upper bounds on $\zeta_{2j}(G)$ when $\zeta_{2j}(G)$ is finite, and  show that $\zeta_{2j}(G)$ is finite for all $r$-regular $G$, $r \geq 2$. 
Appealing to classical results on the factors of cubic graphs,  we prove  that $\zeta_4(G) \leq 2$ for a cubic graph $G$, with equality if and only if $G$ has no 1-factor. We discuss the problem of classifying cubic graphs according to the collection of finite abelian groups  for which they are zero-sum group-magic. 

\end{abstract}

\section{Introduction}
{\rm \ \ } Throughout this paper, graphs will be finite and loopless, but may have multiple edges.  The  vertex set and  edge set of  graph $G$ will be denoted $V(G)$ and  $E(G)$,  respectively.  An edge in $E(G)$ is a bridge if and only if its deletion results in a graph having precisely one more component than $G$ has; the set of bridges of $G$ shall be denoted $B(G)$.  For positive integer $k$, $G$ is    $k$-edge-connected if and only if   $G$ is connected and the deletion of any $k-1$ edges from $E(G)$ does not result in a disconnected graph. 

The set of all non-trivial abelian groups $\mathbb{A} = (A,+)$ will be denoted ${\cal A}$, and  the identity element of each group in  ${\cal A}$ will be denoted by 0. 

Let $G = \big(V(G),E(G)\big)$ be a graph and  let $\mathbb{A} = (A,+) \in {\cal A}$.  Then an {\it $\mathbb{A}$-labeling of $G$} is a function $\phi$ from $E(G)$ into $A - \{0\}$.  For fixed $e \in E(G)$,  $\phi(e)$ is called the {\it label of $e$ under $\phi$}, and  for fixed $v \in V(G)$,   the {\it weight of $v$ under $\phi$} is the sum of the labels of the edges incident to $v$. The graph $G$ is  {\it $\mathbb{A}$-magic} if and only if there exists  an $\mathbb{A}$-labeling $\phi$ of $G$ and an $a \in A$ such that the weight of every vertex in $V(G)$ under $\phi$ is $a$. In such a case, $\phi$ is called an {\it $\mathbb{A}$-magic labeling of $G$}.  Moreover, $G$ is  {\it zero-sum $\mathbb{A}$-magic} if and only if there is an $\mathbb{A}$-labeling $\phi$ of $G$  such that the weight of every vertex in $V(G)$ under $\phi$ is 0. In this case, $\phi $ is  called a {\it zero-sum $\mathbb{A}$-magic labeling of $G$}.  We observe that if $\mathbb{H}$ is a non-trivial subgroup of abelian group $\mathbb{A}$ such that $G$ is zero-sum $\mathbb{H}$-magic, then $G$ is zero-sum $\mathbb{A}$-magic.

Zero-sum group-magic graphs reside in the broader class of magic graphs whose genesis  is due to  Sedl\`a\v{c}ek \cite{sedlacek}. Since his work, variants of group-magic labelings have appeared, including     edge-magic, vertex-magic, total-magic, semi-magic, pseudomagic, and supermagic; see    \cite{gallian} and \cite{wallis}.  Recent works, such as  
\cite{Zk3reg}, \cite{Z2kmagic}, \cite{V4magic}, \cite{old15}, \cite{old16}, \cite{old17}, \cite{old18}, \cite{old22}, \cite{old23}, \cite{old26}, have focussed upon  both  $\mathbb{Z}$-magic and $\mathbb{Z}_j$-magic labelings, where $\mathbb{Z}$ denotes the group of integers under addition and $\mathbb{Z}_j$ denotes the group of integers under modulo $j$ addition. This has led  to  the invention of the integer-magic spectrum of a graph:
\vskip 5pt

\begin{definition} {\rm \ The }zero-sum integer-magic spectrum of graph $G$, {\rm denoted $zim(G)$}, is the set of positive integers  such that 

(a) $1 \in zim(G)$ if and only if $G$ is zero-sum $\mathbb{Z}$-magic, and 

(b)  for $j \geq 2$, $j \in zim(G)$ if and only if $G$ is zero-sum $\mathbb{Z}_j$-magic. 
\end{definition}
\vskip 5pt

\noindent The analysis found in   \cite{Zk3reg} results in the full characterization of  the zero-sum integer-magic spectra of cubic graphs.

\begin{theorem} \label{AAA}     \cite{Zk3reg}  Let $G$ be a cubic graph. If $G$ has a 1-factor, then $zim(G) = \mathbb{N} - \{2\}$.  If $G$ does not have  a 1-factor, then $zim(G) = \mathbb{N} - \{2,4\}$.
\end{theorem}

Additionally, \cite{Zk3reg} continues a discussion of zero-sum $\mathbb{Z}_2^{k_1} \times \mathbb{Z}_4^{k_2}$-magicness, a discussion that was begun in \cite{V4magic} and followed by \cite{Z2kmagic}.  Results from the latter two works include the following.

\begin{theorem} \label{JJJ} \cite{Z2kmagic} If $G$ is a 2-edge-connected graph, then  $G$ is 
 zero-sum $\mathbb{Z}_2^k$-magic for some $k \in \{1,2,3\}$. 
\end{theorem}

\begin{theorem} \label{BBB} \cite{V4magic} Let $G$ be a cubic graph. Then   $G$ is zero-sum $\mathbb{Z}_2^2$-magic if  and only if the chromatic index of $G$ is 3. 

\end{theorem}

\begin{theorem} \label{XXX} \cite{V4magic}
Let $G$ be a graph with a bridge. Then for all positive integers $k$, $G$ is not zero-sum $\mathbb{Z}_2^k$-magic.  
\end{theorem}

\noindent Since these results pertain to direct products whose factors are  $\mathbb{Z}_{2j}$ for $j = 1$, it is  natural to inquire whether or not for any $j > 1$ and any graph $G,$ there exists  a positive integer $k$ such that $G$ is zero-sum $\mathbb{Z}_{2j}^k$-magic. For $j \geq 1$, let $\zeta_{2j}(G)$ be the smallest positive integer $c$ such that $G$ is zero-sum $\mathbb{Z}_{2j}^c$-magic if $c$ exists; infinity otherwise. We note that $\zeta_{2j}(G) \leq k_0$ if and only if $G$ is zero-sum  $\mathbb{Z}_{2j}^k$-magic  for $k \geq k_0$.

In this paper, we consider  zero-sum $\mathbb{A}$-magic labelings of graphs,  with particular attention given to $\mathbb{A} = \mathbb{Z}_{2j}^k$.   In Section 2, we give definitions and foundational theorems.  In Section 3,   we give necessary and sufficient conditions for the finiteness of  $\zeta_{2j}(G)$ where $j \geq 2$. For finite $\zeta_{2j}(G)$, we  show $\zeta_{2j}(G) \leq 6$ for even $j$, and otherwise, $\zeta_{2j}(G) \leq 3 + \vert B(G) \vert$ for general $j$. We also show that $\zeta_{2j}(G)$ is finite for all $r$-regular $G$, $r \geq 2$. 
In Section 4, we extend the results on cubic graphs appearing in \cite{Zk3reg} and \cite{Z2kmagic}, establishing  that $\zeta_4(G) \leq 2$, with equality if and only $G$ has no 1-factor. And in Section 5, we discuss the collection of finite  non-trivial abelian groups $\mathbb{A} $ for which a given cubic graph $G$ is zero-sum $\mathbb{A}$-magic. Since a graph $G$ is zero-sum $\mathbb{A}$-magic if and only if each of its components is zero-sum $\mathbb{A}$-magic, we may base various analyses on connected $G$ with no loss of generality.

%****************SECTION 2***************

\section{Definitions and Preliminary Results}

{\rm \ \ } Let $G$ be a 2-edge-connected graph with $2 \leq \delta(G) < \Delta(G) = 3$   and  vertex $v$ of degree 2. Then the graph that results by {\it smoothing $v$}, denoted $s_G(v)$, is the graph that is produced by  replacing $v$ and its two incident edges with an edge between the neighbors of $v$.  The graph that results by iteratively smoothing each vertex of degree 2 will be denoted $s(G)$. We note that $s(G)$ is a 2-edge-connected cubic graph.

Let $m$ denote a positive integer and let $G$ be a graph. Then  the  {\it m-subdivision of  $G$} is the graph  that results by inserting $m$ distinct vertices along each edge in $E(G)$.  

The {\it martini graph} shall refer to  the graph given in Figure 1. In the sequel, we will have occasion to identify the vertices of degree 1 of  $m$ vertex-disjoint copies of the martini graph with $m$ distinct  vertices of a given graph $G$. In Figure 2, we illustrate the graph $G_0$ that results when $m$ is 3 and $G$ is $K_3$.

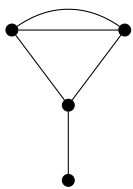
\begin{figure}[h]
\begin{center}
\begin{tikzpicture}[]
\tikzstyle{vertex}=[circle, draw, inner sep=0pt, minimum size=6pt]
\tikzset{vertexStyle/.append style={rectangle}}
	\vertex (1) at (0.25,0) [fill, scale=.75] {};
	\vertex (2) at (1.75,0) [fill, scale=.75] {};
	\vertex (3) at (1,-1) [fill, scale=.75]{};
	\vertex (4) at (1, -2) [fill, scale=.75]{};
	\path
		(1) edge (2)
		(1) edge[bend left=35] (2)
		(1) edge (3)
		(2) edge (3)
		(3) edge (4)

	;
\end{tikzpicture}
\end{center}
\caption{The martini graph.}
\end{figure}

\begin{figure}[h]
\begin{center}
\begin{tikzpicture}[]
\tikzstyle{vertex}=[circle, draw, inner sep=0pt, minimum size=6pt]
\tikzset{vertexStyle/.append style={rectangle}}
	\vertex (1) at (0.75,0.75) [fill, scale=.75] {};
	\vertex (2) at (2.25,0.75) [fill, scale=.75] {};
	\vertex (3) at (1.5,1.5) [fill, scale=.75]{};
	\vertex (4) at (1.5, 2.5) [fill, scale=.75]{};
	\vertex (5) at (2.25, 3.5) [fill, scale=.75]{};
	\vertex (6) at (.75, 3.5) [fill, scale=.75]{};
	\vertex (7) at (-.25, .75) [fill, scale=.75]{};
	\vertex (8) at (-1.25, 1.5) [fill, scale=.75]{};
	\vertex (9) at (-1.25, 0) [fill, scale=.75]{};
	\vertex (10) at (3.25, .75) [fill, scale=.75]{};
	\vertex (11) at (4.25, 1.5) [fill, scale=.75]{};
	\vertex (12) at (4.25, 0) [fill, scale=.75]{};

	\path
		(1) edge (2)
		(2) edge (3)
		(1) edge (3)
		(3) edge (4)
		(4) edge (5)
		(4) edge (6)
		(5) edge (6)
		(5) edge[bend right=35] (6)
		(1) edge (7)
		(7) edge (8)
		(7) edge (9)
		(8) edge (9)
		(8) edge[bend right=35] (9)
		(2) edge (10)
		(10) edge (11)
		(10) edge (12)
		(11) edge (12)
		(11) edge[bend left=35] (12)

	;
\end{tikzpicture}
\end{center}
\caption{ The graph $G_0$ }
\end{figure}
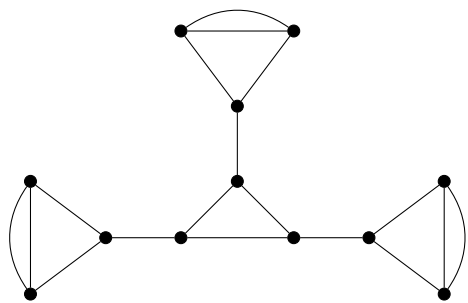

Let $G$ be a  graph with non-empty bridge set $B(G)$. For each edge  $e$ in  $B(G)$, there exist distinct components $H_i$ and $H_j$ of $G - B(G)$ such that $e$ is incident to some vertex in $V(H_i)$ and some vertex in $V(H_j)$. In such a case, we will say that {\it $e$ is incident to  $H_i$ and $H_j$}. 

Let $G$ be a connected graph with $\vert B(G) \vert = b \geq 0$,  and let $H_0, H_1, ..., H_b$ be the components of $G - B(G)$.   Then  $T_G$ shall denote the simple graph with vertex set $V(T_G) = \{h_0, h_1, h_2, ..., h_b\}$ and edge set  $E(T_G)= \{h_ih_j \big\vert $ some edge $e \in B(G)$ is incident to $H_i$ and $H_j\}$. We observe that $T_G$ is a tree, and that  $G$ has no bridges if and only if  $T_G$ is isomorphic to $K_1$.
Furthermore, $G - B(G)$  is a graph with $b + 1$ components $H_0, H_1, H_2,  ..., H_b$ such that each component is either
\vskip 5pt
\indent Type I:  isomorphic to $K_1$, or
\vskip 5pt
\indent Type II:  isomorphic to the $m$-cycle $C_m$ for some $m \geq 2$, or
\vskip 5pt
\indent Type III: isomorphic to some 2-edge-connected graph $H$ with $\Delta(H) = 3$.
\vskip 5pt

\noindent We note that if $H_i$ is of Type I or Type II, then $b > 0$ and $h_i$ is an interior vertex of $T_G$.  Otherwise, if $H_i$ is of Type III, then  $h_i$ is a leaf of $T_G$ if and only if either $b = 0$ or $H_i$ has precisely one vertex of degree 2. Moreover, if $b > 0$ and  $H_i$ is of Type III, then $s(H_i)$  is a 2-edge-connected cubic graph.  To illustrate, we observe that for graph $G_0$ of Figure 2, $T_{G_0}$ is isomorphic to $K_{1,3}$ and $G_0 - B(G_0)$ has one component of Type II and three components of Type III.

It is clear that if $G_1$ and $G_2$ are two connected graphs  such that $T_{G_1}$ is isomorphic to $T_{G_2}$, then $G_1$ and $G_2$ do not necessaily share zero-sum labelability.  For example, consider the graphs $G_1$ and $G_2$ depicted in Figure 3 and Figure 4. One can verify that the edge labels assigned in Figure 3 constitute a zero-sum $\mathbb{Z}_4$-magic labeling of $G_1$, and that $T_{G_1}$ and $T_{G_2}$ are isomorphic to $K_{1,3}$. However, no zero-sum $\mathbb{Z}_4$-magic labeling exists for $G_2$ since Theorem  \ref{YYY} (at the end of this section) guarantees that every zero-sum $\mathbb{Z}_4$-magic labeling of a cubic graph will assign 2 to each bridge.

\begin{figure}[h!]
\begin{center}
\begin{tikzpicture}[]
\tikzstyle{vertex}=[circle, draw, inner sep=0pt, minimum size=6pt]
\tikzset{vertexStyle/.append style={rectangle}}
	\vertex (1) at (-.5,0) [fill, scale=.75]{};
	\vertex (2) at (1,0) [fill, scale=.75] {};
	\vertex (3) at (-.5,1.5) [fill, scale=.75] {}; 
	\vertex (4) at (1,1.5) [fill, scale=.75] {}; 
	\vertex (5) at (1.75,.75) [fill, scale=.75] {}; 
	\vertex (6) at (2.75, .75)[fill, scale=.75]{};
	\vertex (7) at (3.75, .75)[fill, scale=.75]{};
	\vertex (8) at (4.75, .75)[fill, scale=.75]{};
	\vertex (9) at (5.5, 0)[fill, scale=.75]{};
	\vertex (10) at (7,0)[fill, scale=.75]{};
	\vertex (11) at (5.5, 1.5)[fill, scale=.75]{};
	\vertex(12) at (7, 1.5)[fill, scale=.75]{};
	\vertex(13) at (3.25,1.5)[fill, scale=.75]{};
	\vertex(14) at (3.25, 2.25)[fill, scale=.75]{};
        \vertex(15) at (2.5, 3)[fill, scale=.75]{};
	\vertex(16) at (2.5, 4.5)[fill, scale=.75]{};
	\vertex(17) at (4, 3)[fill, scale=.75]{};
	\vertex(18) at (4, 4.5)[fill, scale=.75]{};
	\draw (1) to node[below] {1} (2);
	\draw (1) to node[left] {1} (3);
	\draw (3) to node[above] {1} (4);
	\draw (1) to node[near start, above] {2} (4);
	\draw (2) to node[near start, above] {2} (3);
	\draw (4) to node[above] {1} (5);
	\draw (2) to node[below] {1} (5);
	\draw (5) to node[below] {2} (6);
	\draw (6) to node[below] {1} (7);
	\draw (7) to node[below] {2} (8);
	\draw (8) to node[below] {1} (9);
	\draw (9) to node[below] {1} (10);
	\draw (10) to node[near start, above] {2} (11);
	\draw (10) to node[right] {1} (12);
	\draw (11) to node[above] {1} (12);
	\draw (9) to node[near start, above] {2} (12);
	\draw (8) to node[above] {1} (11);
	\draw (6) to node[left] {1} (13);
	\draw (7) to node[right] {1} (13);
	\draw (13) to node[right] {2} (14);
	\draw (14) to node[left] {1} (15);
	\draw (15) to node[left] {1} (16);
	\draw (14) to node[right] {1} (17);
	\draw (17) to node[right] {1} (18);
	\draw (16) to node[above] {1} (18);
	\draw (17) to node[near start, above] {2} (16);
	\draw (18) to node[near end, above] {2} (15);

\end{tikzpicture}
\end{center}
\caption{The graph $G_1$ with a zero-sum $\mathbb{Z}_4$-magic labeling.}
\label{fig:tree_ex1}
\end{figure}
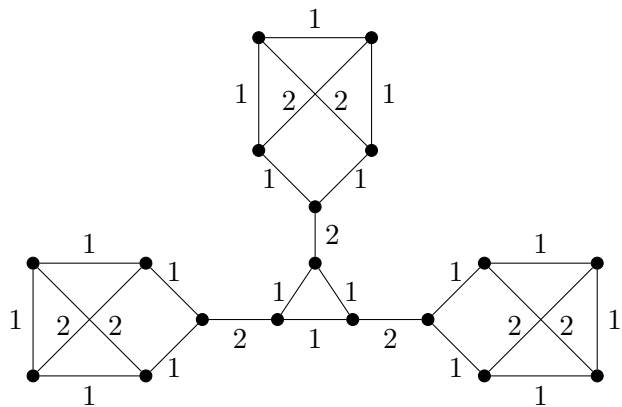

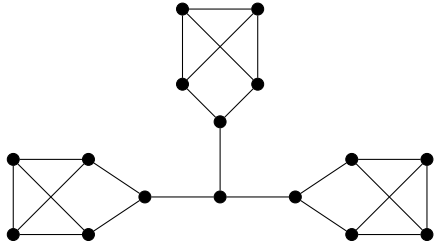
\begin{figure}[h!]
\begin{center}
\begin{tikzpicture}[]
\tikzstyle{vertex}=[circle, draw, inner sep=0pt, minimum size=6pt]
\tikzset{vertexStyle/.append style={rectangle}}
	\vertex (1) at (0,0) [fill, scale=.75] {};
	\vertex (2) at (1,0) [fill, scale=.75] {};
	\vertex (3) at (0,1) [fill, scale=.75] {}; 
	\vertex (4) at (1,1) [fill, scale=.75] {}; 
	\vertex (5) at (1.75,.5) [fill, scale=.75] {}; 
	\vertex (6) at (2.75, .5)[fill, scale=.75]{};
	\vertex (7) at (3.75, .5)[fill, scale=.75]{};
	\vertex (8) at (4.5, 0)[fill, scale=.75]{};
	\vertex (9) at (5.5, 0)[fill, scale=.75]{};
	\vertex (10) at (4.5,1)[fill, scale=.75]{};
	\vertex (11) at (5.5, 1)[fill, scale=.75]{};
	\vertex(12) at (2.75, 1.5)[fill, scale=.75]{};
	\vertex(13) at (2.25,2)[fill, scale=.75]{};
	\vertex(14) at (2.25, 3)[fill, scale=.75]{};
	\vertex(15) at (3.25, 2)[fill, scale=.75]{};
	\vertex(16) at (3.25, 3)[fill, scale=.75]{};
	\path
		(1) edge (2)
		(1) edge (3)
		(1) edge (4)
		(2) edge (3)
		(3) edge (4)
		(2) edge (5)
		(4) edge (5)
		(5) edge (6)
		(6) edge (7)
		(7) edge (8)
		(8) edge (9)
		(7) edge (10)
		(10) edge (11)
		(9) edge (11)
		(11) edge (8) 
		(9) edge (10)
		(6) edge (12)
		(12) edge (13)
		(14) edge (13)
		(12) edge (15)
		(16) edge (15)
		(14) edge (16)
		(13) edge (16)
		(14) edge (15)

	;
\end{tikzpicture}
\end{center}
\caption{The graph $G_2$ with no zero-sum $\mathbb{Z}_4$-magic  labeling.}
\label{fig:tree_ex2}
\end{figure}

There is a large body of work pertaining to factors of graphs, some of which address 1-factors of  cubic graphs. Results of use to this paper include the following.

\begin{theorem} \label{PQR} \cite{Petcube}   Every 2-edge-connected, cubic graph contains a $1$-factor. 
\end{theorem}

\begin{theorem} \label{STU} \cite{Petcube} Every connected cubic graph with at most two bridges contains a $1$-factor. 
\end{theorem}

\begin{theorem} \label{MLK} \cite{West} Let $G$ be a 2-edge-connected cubic graph and let $e$ be a fixed edge in $E(G)$. Then 

(a) there exists a 1-factor of $G$ that contains $e$, and 

(b) there exists a 2-factor of $G$ that contains $e$.
\end{theorem}

\begin{theorem}\label{ABC}\cite{regfactors}\label{thm:2factors}
 Let $G$ be a $k$-regular, $(k-1)$-edge-connected graph with an even number of vertices, and let $m$ be an integer such that $1 \le m \le k-1$. Then the graph obtained by removing any $k-m$ edges of $G$ has an $m$-factor. 
\end{theorem} 

\begin{theorem} \label{YYY}  Let $m$ denote a positive integer  and let $f$  be a function from $\{1,2,3,..., m\}$ into $\mathbb{N}$. Let $G$ be a cubic graph with bridge $e_0$ and let $\phi$ be a zero-sum $\Pi_{i=1}^m\mathbb{Z}_{2f(i)}$-magic labeling of $G$. Then the $i^{th}$ coordinate of $\phi(e_0)$ is in $\{0, 2, 4, 6, ..., 2f(i)-2\}$.

\end{theorem}

\begin{proof} Let $v$ denote a terminal vertex of $e_0$ and let $H$ denote the component of $G - \{e_0\}$ such that   $v \in V(H)$. Then under $\phi$, the sum (mod $2f(i)$) of the $i^{th}$ coordinates of the  weights of the vertices in $V(H)$  is the $i^{th}$ coordinate of the sum $\phi(e_0) + \sum_{e \in E(H)} 2\phi(e)$. Since this sum is 0 and since $2f(i)$ is even, the result follows.  
\end{proof}

\section{On zero-sum $\mathbb{Z}_{2j}^k$-magic graphs}

\begin{theorem} \label{RRR} Let $j \geq 2$ and let $G$ be a connected graph with non-empty bridge  set $B(G)$.  If  the removal of any one bridge of $G$ does not result in a graph with a component that is either  bipartite or isomorphic to $K_1$, then   $\zeta_{2j}(G) \leq 3 + \vert B(G) \vert$. Otherwise, there is  no finite non-trivial abelian group $\mathbb{A}$ such that $G$ is zero-sum $\mathbb{A}$-magic. 
\end{theorem}

\begin{proof}  Let $B(G) = \{b_1, b_2, ..., b_m\}$. 

Suppose that  $b_i \in B(G)$  such that $G - \{b_i\}$ has either a trivial component or a bipartite component. If $G - \{b_i\}$ has a trivial component, then $G$ has a vertex of degree 1, implying that    $G$ is zero-sum $\mathbb{A}$-magic for no finite non-trivial abelian group $\mathbb{A}$. If  $G - \{b_i\}$ has a bipartite component $H$, we suppose that the parts of $H$  are $X$ and $Y$ such that, with no loss of generality, $b_i$ is incident to $x \in X$.  Let $E_Y$ and $E_X$ respectively denote the set of  edges  of $G$ that are incident to some vertex in $Y$ and some vertex in  $X$. (Note that $E_Y = E(H)$ and $E_X= E(H) \bigcup \{b_i\}$.) Then, supposing  the contrary that $G$ has a zero-sum $\mathbb{A}$-magic labeling $\phi$  for finite non-trivial abelian group $\mathbb{A}$, it follows that  under $\phi$, the sum of the labels of the edges  in $E_Y$ and the sum of the labels of the edges  in $E_X$ must each be 0. We thus have the contradiction that $\phi(b_i) = 0$.

Suppose there is no bridge  $b_i$  of $G$ such that $G - \{b_i\}$ has either a trivial component or a bipartite component. Then each non-trivial component of $G-B(G)$  is 2-edge-connected, implying that each non-trivial component admits (by Theorem \ref{JJJ})  a zero-sum $\mathbb{Z}_2^3$-magic labeling $\phi$.  Since $j\phi$ (base 10 multiplication) is thus a zero-sum $\mathbb{Z}_{2j}^3$-labeling,  it follows that $G - B(G)$ is zero-sum $\mathbb{Z}_{2j}^3$-magic for each $j \geq 2$.  

For fixed arbitrary $j_0 \geq 2$, let $\phi'$ denote a zero-sum $\mathbb{Z}_{2j_0}^3$-magic  labeling of $G-B(G)$. We proceed to  construct a  zero-sum $\mathbb{Z}_{2j_0}^{3 + m}$-magic labeling  $\phi''$ of $G$. 

Select arbitrary $r$, $1 \leq r \leq m$, and consider   $b_r = \{x_r,y_r\} $ where respectively, $x_r$ and $y_r$ are vertices in distinct components  $X_r$ and $Y_r$ of $G-\{b_r\}$. Since neither $X_r$ nor $Y_r$ is bipartite or trivial, each contains an odd cycle: $C_{X_r}$ and $C_{Y_r}$, respectively.  If $x_r$ is not incident to $C_{X_r}$, let $P_{X_r}$  be a shortest path in $X_r$ from $x_r$ to $C_{X_r}$. Similarly, if $y_r$ is not incident to $C_{Y_r}$, let $P_{Y_r}$  be a shortest path in $Y_r$ from $y_r$ to $C_{Y_r}$.  

Consider the subgraph $G_r$ of $G$ induced by $b_r$ and the edges of $C_{X_r}$, $P_{X_r},$ $C_{Y_r}$,  and $P_{Y_r}$. We establish a zero-sum $\mathbb{Z}_{2j_0}$-magic labeling $\phi_r$ of $G_r$.

\noindent Case 1. $j_0$ is even.  Assign $j_0$ to $b_r$ and each edge along  $P_{X_r}$  and  $P_{Y_r}$. It is easy to see that labels of ${j_0 \over 2}$ and ${3j_0 \over 2}$ can be deployed about the edges of $C_{X_r}$ and $C_{Y_r}$ such that  the  weight of each vertex in $V(G_r)$ under $\phi_r$  is 0. 

\noindent Case 2. $j_0$ is odd. Assign $j_0-1$ to $b_r$. To each edge incident to $P_{X_r}$  or  $P_{Y_r}$, assign $j_0-1$ or $j_0 + 1$ as appropriate.  It is easy to see that labels of ${j_0 -1\over 2}$, ${j_0 + 1 \over 2}$,   ${3j_0 -1 \over 2}$, and ${3j_0 + 1 \over 2}$  can be deployed about the edges of $C_{X_r}$ and $C_{Y_r}$ such that  the  weight of each vertex in $V(G_r)$ under $\phi_r$  is 0. 

We now establish a zero-sum $\mathbb{Z}_{2j_0}^{3 + m}$-magic labeling  $\phi''$ of $G$ as follows:  Let  $e \in E(G).$  For $1 \leq i \leq m$,  let the $i^{th}$ coordinate of $\phi''(e)$ be $\phi_i(e)$ if $e \in E(G_i)$; 0 otherwise.  For $1 \leq i \leq 3$, let the $(m + i)^{th}$ coordinate of $\phi''(e)$ be the $i^{th}$ coordinate of $\phi'(e)$ if $e$ is in $G - B(G)$;  0 otherwise. 
\end{proof}

Suppose $G$ is a a connected   $r$-regular graph  for $r \geq 2$ and suppose $j \geq 2$.   If $G$ has no bridges, then by Theorem \ref{JJJ}, $G$ is zero-sum $\mathbb{Z}_{2j}^k$-magic  for some $k \in \{1,2,3\}$. If $G$ has at least one bridge, then $r$ is odd, and hence the removal of any bridge results in two components each with an odd cycle.  Therefore, by Theorem \ref{RRR}, we have the following corollary. 

\begin{corollary} \label{TTT} For every $j, r \geq 2$, $\zeta_{2j}(G) \leq 3 +\vert B(G) \vert$ for every  $r$-regular graph $G$ (connected or not). 
\end{corollary}

In the next theorem, we give necessary and sufficient conditions for $G$  to be  zero-sum $\mathbb{Z}_{2j}^6$-magic in the case that $j$ is even. 

\begin{theorem} \label{SSS} Let $j$ be a positive even integer  and let $G$ be a connected graph with non-empty bridge  set $B(G)$.  If  the removal of any one bridge of $G$ does not result in a graph with a component that is either  bipartite or isomorphic to $K_1$, then $\zeta_{2j}(G) \leq 6$. 
 Otherwise, there is  no finite non-trivial abelian group $\mathbb{A}$ such that $G$ is zero-sum $\mathbb{A}$-magic. 
\end{theorem}

\begin{proof} Suppose the removal of a bridge of $G$ results in an isolated vertex or a bipartite component. Then per the argument given in Theorem \ref{RRR}, there is no finite non-trivial abelian group $\mathbb{A}$ such that $G$ is zero-sum $\mathbb{A}$-magic.

Suppose the removal of a bridge of $G$ results in no isolated vertex or  bipartite component. If $\vert B(G)\vert = m \leq 3$, then by Theorem \ref{RRR}, $G$ is zero-sum $\mathbb{Z}_{2j}^{3+m}$-magic, implying that $G$ is zero-sum $\mathbb{Z}_{2j}^{6}$-magic. We therefore assume $\vert B(G) \vert \geq 4$.

The tree $T_G$ has at least two leaves, each of which represents a component of $G-B(G)$. Let $H_0, H_1, H_2, ..., H_s$ be the components of $G-B(G)$ that are represented by leaves of $T_G$, and let $b_i = \{x_i, y_i\}$ denote the bridge of $G$ that is incident to $H_i$, where $y_i$ is the vertex in $H_i$. (The $x_i$s may not be pairwise distinct.) For each $i$, $0 \leq i \leq s$, $G - \{b_i\}$ results in two components, one of which is $H_i$ which in turn must contain an odd cycle $C_i$ by hypothesis. If $y_i$ is not incident to $C_i$, there is a shortest path $P_i$ from $y_i$ to $C_i$ in $H_i$.

Let $G^*$ be the subgraph of $G$ that results by deleting all of the vertices of each $H_i$ except $y_i$,  and let  $G^{'}$ denote the graph that results by  identifying the vertices $y_0, y_1, y_2, ..., y_s$ in $G^*$ to a single vertex $y'$.  We observe that each bridge $b_i$,  $0 \leq i \leq s$, induces an edge $b_i' = \{x_i, y'\}$ in $G'$ and a bridge $b_i = \{x_i, y_i\}$ in $G^*$ where $y_i$ has degree 1  in $G^*$.

By Theorem \ref{JJJ}, there exists a zero-sum $\mathbb{Z}_2^3$-magic labeling  of $G'$ which, upon multiplication by $j$ (base 10),   results in  a zero-sum  $\mathbb{Z}_{2j}^3$-magic labeling $\phi'$ of  $G'$.  We let $\phi^*$ denote an edge-labeling of $G^*$  such that

\vskip 5pt

: if $e$ is an edge common to $G'$ and $G^*$, then $\phi^*(e) = \phi'(e)$;

: if $e = b_i$, then $\phi^*(e) = \phi'(b_i')$.

\vskip 5pt

\noindent It is clear that except for each $y_i$ of $G^*$, each vertex of $G^*$ has incident edges whose labels under $\phi^*$ sum to 0.  Moreover, due to the base 10 multiplication upon which $\phi'$ is founded, every  label under $\phi^*$ of each $e \in E(G^*)$ has coordinates of 0 or $j$,    with at least one coordinate equalling $j$. 

Let $G^{**}$ be the subgraph of $G$ induced by  precisely the edges of $G^*$ and the edges of each $P_i $ and  $C_i$. Let $\phi^{**}$ be an edge labeling of $G^{**}$ such that 

: if $e$ is not among the edges of any $P_i$ or $C_i$, $\phi^{**}(e) = \phi^*(e)$, and

: if $e$ is an edge of $P_i$, then $\phi^{**}(e) = \phi^*(b_i)$, and

: if $e$ is an edge of $C_i$, then $\phi^{**}(e) = (a,b,c)$ where $(a,b,c) \neq (0,0,0)$ and $a,b,c   \in \{0, {j \over 2}, {3j \over 2}\}$. It is easy to see that  the labels can be so assigned such that  each vertex incident to $C_i$ has incident edges whose labels sum to 0. Thus, $\phi^{**}$ is a zero-sum $\mathbb{Z}_{2j}^3$-magic labeling of $G^{**}$.

For each $i$, $0 \leq i \leq s$, let $\phi_i$ be a zero-sum $\mathbb{Z}_{2j}^3$-magic labeling of the  2-edge-connected graph $H_i$. We complete the construction of $\phi$ as follows:

: if  $e$ is an edge common to  $G^*$ and $G^{**}$, the first three coordinates of $\phi(e)$ are respectively the three coordinates  of $\phi^{**}(e)$ and the last three coordinates  are each 0;

: if $e$  is in the edge set of some $H_i$ but is not an edge incident to $P_i$ or $C_i$, the last three coordinates  of $\phi(e)$ are respectively the three coordinates  of  $ \phi_i(e)$ and the first three coordinates  are each 0;

: if $e$ is   in the edge set of some $P_i$ or $C_i$, the last three coordinates  of $\phi(e)$ are respectively the three coordinates  of  $ \phi_i(e),$ and the first three coordinates  are respectively the coordinates  of $\phi^{**}(e)$. 
\end{proof}

\vskip 10pt

In an argument analogous to that of the previous corollary, we have

\begin{corollary} \label{EEE} For every $j, r \geq 2$ where $j$ is even,  $\zeta_{2j}(G) \leq 6$ for every   $r$-regular graph $G$ (connected or not).
\end{corollary}

\section{Cubic graphs are zero-sum $\mathbb{Z}_4^2$-magic}

In this section, we narrow our focus to the zero-sum $\mathbb{Z}_{2j}^k$-magicness of cubic graphs.

By Theorem \ref{AAA},  each  cubic graph $G$ has $zim(G) = \mathbb{N}-\{2\} $ or $\mathbb{N} - \{2,4\}$. Thus, if $j \geq 3$, $G$ is zero-sum $\mathbb{Z}_{2j}^k$-magic for $k \geq 1$. If $j = 1$, then by Theorems \ref{JJJ}, \ref{BBB}, and \ref{XXX}, $G$  is 
\begin{itemize} 

\item zero-sum $\mathbb{Z}_2^k$-magic for precisely $k \geq 2$ if and only if it has  chromatic index 3;  

  \item zero-sum $\mathbb{Z}_2^k$-magic for precisely  $k \geq 3$ if $G$ is 2-edge-connected with chromatic index 4; 

\item zero-sum $\mathbb{Z}_2^k$-magic for no $k$ if $G$ has a bridge.
\end{itemize}
\noindent It thus remains to explore the case $j = 2$.

By Corollary \ref{EEE}, every  cubic graph is zero-sum $\mathbb{Z}_4^k$-magic for $k \geq 6$. Moreover, by Theorems \ref{AAA} and \ref{STU}, every   cubic graph with at most two bridges is zero-sum $\mathbb{Z}_4^k$-magic for all $k \geq 1$. We thus focus on   cubic graphs with at least 3 bridges in the cases $1 \leq k \leq 5$.

We readily adapt the proof of Theorem \ref{SSS} above to establish the following.

\begin{theorem} \label{PPD} For every cubic graph $G$, $\zeta_4(G) \leq 3$.

\end{theorem}
\begin{proof}  It suffices to show that $G$ is zero-sum $\mathbb{Z}_4^3$-magic where $G$ is connected and cubic with at least 3 bridges. 

Define $H_0, H_1, ..., H_s, b_0, b_1, b_2, ..., b_s$, $G^*$, and $G'$ as given in the proof of Theorem \ref{SSS}. (Recall that $H_0, H_1, ..., H_s$ are components of $G - B(G)$ that are represented by the leaves of $T_G$.) Then by Theorem \ref{JJJ} and multiplication by 2 (base 10), there exists a zero-sum $\mathbb{Z}_4^3$-magic labeling $\phi'$ of $G'$.  Let $\phi^*$ denote an edge labeling of $G^*$ such that

: if $e$ is an edge common to $G^*$ and $G'$, then $\phi^*(e) = \phi'(e)$;

: if $e = b_i$, then $\phi^*(e) = \phi'(b_i')$. 

\noindent It is clear that except for each $y_i$ of $G^*$, each vertex of $G^*$ has incident edges whose labels under $\phi^*$ sum to 0. Moreover, for every edge $e$ of $G^*$, each coordinate of $\phi^*(e)$ is 0 or 2, with at least one coordinate equalling 2.

Select arbitrary $r$, $0 \leq r \leq s$. Then $H_r$ is  2-edge-connected and has precisely one vertex $y_r$ of degree 2. Let the neighbors of $y_r$ in $H_r$ be the necessarily distinct vertices $y_r'$ and $y_r''$. Then by smoothing $y_r$, we create a 2-edge-connected cubic graph $s(H_r)$ (possibly with 2 or 3 distinct  edges incident to $y_r'$ and $y_r''$). Thus, by Theorem \ref{MLK}, $s(H_r)$ has a 2-factor $M_r$ containing $y_r'y_r''$. Appealing to the evennesss of the coordinates of $\phi^*(b_r)$, we produce a zero-sum $\mathbb{Z}_4^3$-magic labeling ${\hat \phi}_r$ of $s(H_r)$ by having ${\hat \phi}_r$ assign $\frac{1}{2}\phi^*(b_r)$ to each edge in $M_r$ and $\phi^*(b_r)$ to each edge in $E\big(s(H_r)\big) - M_r$. 

It is now easy to construct a zero-sum $\mathbb{Z}_4^3$-magic labeling of $G$:

for $e \in E(G^*) \bigcap E(G)$, $\phi(e) = \phi^*(G^*)$;

for $e \in E(H_r) \bigcap E\big(s(H_r)\big)$, $\phi(e) = {\hat \phi}_r(e)$;

for $e = y_r'y_r$ or $e = y_r''y_r$, $\phi(e) = {\hat \phi}_r(y_r'y_r'')$.  
\end{proof}
By Theorem \ref{PPD},  each    cubic graph falls into one of three categories: for $1 \leq i \leq 3$, Category $i$ represents the   cubic graphs $G$ such that  $\zeta_4(G) = i$. Our unsuccessful search for graphs of Category 3 among   cubic graphs with necessarily no 1-factor   led us to conjecture the emptiness of that category. We devote the remainder of this section to verifying that conjecture by proving that all cubic graphs are zero-sum $\mathbb{Z}_4^2$-magic.

\begin{lemma} \label{QQQ} Let $k$ denote a positive integer and let $G$ denote a connected graph with maximum degree 3 and minimum degree 2 such that the set $S$ of vertices of $G$ with degree 2 has cardinality $2k$. 
Then  there exist $k$ pairwise vertex-disjoint paths $P_1, P_2, ..., P_k$ in $G$ whose terminal vertices comprise $S$. 
\end{lemma}

\begin{proof} By induction, we show that  for each $h$, $1 \leq h \leq {k}$, there exist $h$  pairwise vertex-disjoint paths $P_1, P_2, ..., P_h$ in $G$ whose terminal vertices comprise a subset of $S$ with cardinality $2h$. 

The base case $h=1$ is clearly established by the connectedness of $G$.   Thus, suppose $h_0$ is an integer, $1 \leq h_0 < k$, such that there exist $h_0$  pairwise vertex-disjoint paths  in $G$ whose terminal vertices comprise a subset of $S$ with cardinality $2h_0$.  Then there exists  ${\cal P}$, a set of $h_0$  pairwise vertex-disjoint paths  $P_1, P_2, ..., P_{h_0}$  whose terminal vertices comprise a subset of $S$ with cardinality $2h_0$  and whose lengths have minimum sum. Let $x$ and $y$ be distinct vertices of degree two in $G$, neither of which is a terminal vertex  of $P_i$, $1 \leq i \leq h_0$. We observe that neither $x$ nor $y$ is incident to some path in ${\cal P}$, since otherwise  the minimality of the sum of the lengths of the paths in ${\cal P}$ is violated.

By the connectedness of $G$, there exists a path $Q$ in $G$ from $x $ to $y$. Consider the graph $G_1$ given by $Q \bigcup  P_1 \bigcup P_2 ... \bigcup P_{h_0}$. Since $G$ has maximum degree 3, it follows that if $Q$ intersects $P_i$ at an interior vertex $v$ of $P_i$, then $Q$ and $P_i$ also share an edge incident to $v$. Similarly, since each terminal vertex  of $P_i$ has degree 2 in $G$, it follows that if $Q$ intersects $P_i$ at a terminal vertex $v$ of $P_i$, then $Q$ and $P_i$ also share an edge incident to $v$. Let $G_2$ be the graph that results by deleting each edge that is common to $Q$ and some $P_i$. We observe that  in $G_2$, all vertices  have degree 2  except $x, y,$ and the terminal points  of the paths in ${\cal P}$ (which have degree 1). 
Thus, $G_2$ consists of components, each of which is either a cycle or a path. Since there are precisely $2h_0+2$ vertices of degree 1 in $G_2$, $h_0+1$ of those components are  pairwise vertex-disjoint paths in $G$ whose terminal vertices comprise a subset of $S$ with cardinality $2h_0 + 2$. 
\end{proof}

A collection of pairwise vertex-disjoint paths  whose existence is guaranteed by Lemma \ref{QQQ} shall be called a {\it threading of G}.

\begin{theorem}\label{PPP} Suppose  $G$ is a  cubic graph. If $G$ has a 1-factor, then $\zeta_4(G) = 1$. Otherwise,   $\zeta_4(G) = 2$. 
\end{theorem}

\begin{proof} By Theorem \ref{AAA} and previous discussion, it suffices to show that if $G$ is connected with at least 3 bridges, then $G$ is zero-sum $\mathbb{Z}_4^2$-magic.  

Denote the components of $G - B(G)$ by $H_0, H_1, H_2, ..., H_b$, where $b \geq 3$ is necessarily the number of bridges of $G$.  Then the vertex set of the tree $T_G$ is $\{h_0, h_1, h_2, ..., h_b\}.$  

We proceed by constructing a zero-sum $\mathbb{Z}_4^2$-magic labeling $\phi$ of $G$. Note that throughout the construction, each bridge is assigned a label from among $(2,2), (2,0)$, and $(0,2)$ as required by Theorem  \ref{YYY}.  

\begin{enumerate}

\item[1.] With no loss of generality, assume vertex  $h_0$ is a leaf of $T_G$ and that the  bridge  of $G$ that is incident to  component $H_0$ is $b_0 = x_0v_0$ where $v_0 \in V(H_0)$. Then $v_0$ is the only vertex of degree $2$ of $H_0$, whose necessarily distinct neighbors (in $H_0$) we denote by $v_0'$ and $v_0''$. Since $s(H_0)$ is a $2$-connected cubic graph, there exists a $2$-factor $M_0$ of $s(H_0)$ containing the edge $v_0'v_0''$ that is induced by smoothing $v_0$.  Let $f_0$ be an edge labeling of $s(H_0)$ such that  each edge of $M_0$ is given label $1$ and all edges of $E\big(s(H_0)\big) - M_0$ are given label $2$. Necessarily, $f_0$ is a zero-sum $\mathbb{Z}_4$-magic labeling of $s(H_0)$. We  commence the construction of our desired zero-sum $\mathbb{Z}_4^2$-magic labeling $\phi$ of $G$ by assigning labels to $E(H_0) \bigcup \{b_0\}$ as follows: 

\[ \phi(e) = \begin{cases}   \big(f_0(e),0 \big) & \text{if    $e \in E\big(s(H_0)\big) \cap E(H_0)$ } \\ (1,0) & \text{if $e = v_0v_0'$ or $v_0v''_0$} \\ (2,0) &\text{if   $e = b_0$}  \end{cases}\]

\noindent Observe that for each vertex $v$ of degree 3 in the subgraph of $G$ induced by $E(H_0) \bigcup \{b_0\}$, the sum of the labels under $\phi$  of the edges incident to $v$ is 0.

\item[2.] Until every edge in $E(G)$ is labeled under $\phi$, do the following:

Select a component $H_j$ such that (1) the edges of $H_j$ are as yet unlabeled under $\phi$, and (2) some  bridge of $G$  is labeled under $\phi$ and is incident to  $H_j$.  Let  $b_{j,1}, b_{j,2}, \dots, b_{j, k}$ be the bridges of $G$ that are incident to $H_j$, and with no loss of generality,  suppose that $b_{j,1}$ is the necessarily unique bridge among them that is labeled under $\phi$. 

The component $H_j$ is characterized by one of the following:

a) $H_j$ is isomorphic to $K_1$;

b) $H_j$ is isomorphic to the $k$-cycle $C_k$, $k \geq 2$; 

c) $H_j$  has minimum degree 2 and maximum degree 3. 

Suppose a) holds. If $\phi(b_{j,1}) = (2,0)$ or $(0,2)$ or $(2,2)$, it is an easy matter to extend $\phi$ in such a way that under $\phi$, the bridges $b_{j,2}$ and $b_{j,3}$, each incident to the sole vertex $v_{j,1}$ of $H_j$,  are assigned  labels from  among $(0,2), (2,0), (2,2)$ so that  $\sum_{i=1}^3 \phi(b_{j,i}) = 0$. 

Suppose b)  holds. If $\phi(b_{j,1}) = (2,0)$ or $(0,2)$ or $(2,2)$, it is an easy matter to extend $\phi$ in such a way that for each vertex $v_{j,i}$ of the cycle, its three incident edges (which include one bridge of $G$) are assigned   labels from  among $(0,2), (2,0), (2,2)$ so that the sum of the labels of the edges incident to $v_{j,i}$ is 0.  

Now assume that c) holds. Let $v_{j,1},\dots, v_{j,k}$ denote the (necessarily distinct) vertices of degree $2$ in $H_j$ such that  the bridge $b_{j,i}$ of $G$ is incident to $v_{j,i}$. Also  let $v_{j,i}'$ and $v_{j,i}''$ be the (necessarily distinct)  neighbors of $v_{j,i}$.
\begin{enumerate}
\item[(i)] Suppose $k$ is even. We proceed by considering the label assigned to $b_{j,1}$.
\begin{itemize}
\item If $\phi(b_{j,1}) = (2,0) $ or $(0,2)$:

Since $s(H_j)$ is a $2$-connected cubic graph, there exists a 1-factor $M_j$ of $s(H_j)$ containing the edge of the form  $v_{j,1}'v_{j,1}''$. Let $f_j$ be an edge-labeling of $s(H_j)$ where each edge of $M_j$ is given label $2$ and all edges of $E\big(s(H_j)\big) - M_j$ are given label $1$. Necessarily, $f_j$ is a zero-sum $\mathbb{Z}_4$-magic labeling of $s(H_j)$. 

Next, choose a threading $Th(j)$ of $H_j$. Let $g_j$ be an edge-labeling of $H_j$ such that $g_j$ assigns 2 to each edge  that lies along some path of  $Th(j)$ and 0 to each edge that lies along no path of $Th(j)$.

If $\phi(b_{j,1}) = (2,0)$,  extend $\phi$ to the edges in $E(H_j) \bigcup \{b_{j,2}, ..., b_{j,k}\}$ as follows:

\[ \phi(e) = \begin{cases} \big(g_j(e), f_j(v_{j,i}'v_{j,i}'') \big) & \text{if $e = v_{j,i}v_{j,i}'$ or $e=v_{j,i}v_{j,i}''$, $1 \le i \le k$}\\  \big(g_j(e), f_j(e) \big) & \text{if $e \in s(H_j)$}  \\ (2,0) & \text{if $e = b_{j,i}$ and $f_j(v_{j,i}'v_{j,i}'') = 2$, $2 \le i \le k$}\\ (2,2) & \text{if $e = b_{j,i}$ and $f_j(v_{j,i}'v_{j,i}'') = 1$, $2 \le i \le k$}

\end{cases}\]

If $\phi(b_{j,1}) = (0,2)$,  extend $\phi$ to the edges in $E(H_j) \bigcup \{b_{j,2}, ..., b_{j,k}\}$ as follows:
\[ \phi(e) = \begin{cases} \big( f_j(v_{j,i}'v_{j,i}''), g_j(e)\big) & \text{if $e = v_{j,i}v_{j,i}'$ or $e=v_{j,i}v_{j,i}''$, $1 \le i \le k$}\\  \big(f_j(e), g_j(e) \big) & \text{if $e \in s(H_j)$}  \\ (0,2) & \text{if $e = b_{j,i}$ and $f_j(v_{j,i}'v_{j,i}'') = 2$, $2 \le i \le k$}\\ (2,2) & \text{if $e = b_{j,i}$ and $f_j(v_{j,i}'v_{j,i}'') = 1$, $2 \le i \le k$} \end{cases}\]

\item If $\phi(b_{j,1}) = (2,2)$:  

Since $s(H_j)$ is a $2$-connected cubic graph, there exists a $2$-factor $M_j$ of $s(H_j)$ containing the edge of the form  $v_{j,1}'v_{j,1}''$. Let $f_j$ be an edge-labeling of $s(H_j)$ such that each edge of $M_j$ is given label $1$ and all edges of $E\big(s(H_j)\big) - M_j$ are given label $2$. Necessarily, $f_j$ is a zero-sum $\mathbb{Z}_4$-magic labeling of $s(H_j)$. 

Next, choose a threading $Th(j)$ of $H_j$. Let $g_j$ be an edge-labeling of $H_j$ such that $g_j$ assigns 2 to each edge  that lies along some path of  $Th(j)$ and 0 to each edge that lies along no path of $Th(j)$.
Extend $\phi$ to the edges in $E(H_j) \bigcup \{b_{j,2}, ..., b_{j,k}\}$ as follows:
\[ \phi(e) = \begin{cases} \big(f_j(v_{j,i}'v_{j,i}''), g(e)\big) & \text{if $e= v_{j,i}v_{j,i}'$ or $e=v_{j,i}v_{j,i}''$, $1 \le i \le k$}\\   \big(f_j(e), g(e)\big) & \text{if $e \in  E\big(s(H_j)\big)$} \\ (0,2) & \text{if $e = b_{j,i}$ and $f_j(v_{j,i}'v_{j,i}'') = 2$, $2 \le i \le k$}\\ (2,2) & \text{if $e = b_{j,i}$ and $f_j(v_{j,i}'v_{j,i}'') = 1$, $2 \le i \le k$}\end{cases}\]
\end{itemize}

\item[(ii)] Suppose $k = 1$.  Let $M_j$ be a 2-factor of $s(H_j)$ that contains the edge of the  form $v_{j,1}'v_{j,1}''$, and let $f_j$ be an edge labeling of $s(H_j)$ such that $f_j$ assigns 1 to each edge in $M_j$ and 2 to each edge in $E\big(s(H_j)\big) - M_j$. Then $f_j$ is necessarily a zero-sum $\mathbb{Z}_4$-magic labeling of $s(H_j)$, and we extend $\phi$ to $E(H_j)$ according to $\phi(b_{j,1})$ as follows: 

If $\phi(b_{j,1}) = (2,0)$, 

\[ \phi(e) = \begin{cases} \big(f_j(v_{j,1}'v_{j,1}''), 0\big) & \text{if $e = v_{j,1}v_{j,1}'$ or $e=v_{j,1}v_{j,1}''$} \\ \big(f_j(e), 0\big) & \text{if $e \in s(H_j)$}   \end{cases}\]

If $\phi(b_{j,1}) = (0,2)$, 

\[ \phi(e) = \begin{cases}\big(0,  f_j(v_{j,1}'v_{j,1}'')\big) & \text{if $e = v_{j,1}v_{j,1}'$ or $e=v_{j,1}v_{j,1}''$} \\ \big(0, f_j(e)\big) & \text{if $e \in s(H_j)$}   \end{cases}\]

If  $\phi(b_{j,1}) = (2,2)$,

\[ \phi(e) = \begin{cases}\big(f_j(v_{j,1}'v_{j,1}''), f_j(v_{j,1}'v_{j,1}'')\big) & \text{if $e = v_{j,1}v_{j,1}'$ or $e=v_{j,1}v_{j,1}''$} \\ \big(f_j(e), f_j(e)\big) & \text{if $e \in s(H_j)$}   \end{cases}\]

\item[(iii)] Suppose $k \geq 3$ is odd. We proceed by considering the label assigned to $b_{j,1}$.
\begin{itemize}
\item If $\phi(b_{j,1}) = (2,2)$:

Since $s(H_j)$ is a $2$-edge-connected cubic graph,  Theorem \ref{ABC} guarantees the existence of  a $2$-factor $M_j$ containing the (not necessarily distinct) edges $v_{j,1}'v_{j,1}''$ and $v_{j,2}'v_{j,2}''$. Let $f_j$ be an edge-labeling of $s(H_j)$ such that  each edge of $M_j$ is given label $1$ and all edges of $E\big(s(H_j)\big) - M_j$ are given label $2$. Necessarily, $f_j$ is a zero-sum $\mathbb{Z}_4$-magic labeling of $s(H_j)$.
Observing that $s_{H_j}(v_{j,2})$ (the graph that results by smoothing $v_{j,2}$ of $H_j$) has an even number of vertices of degree 2,  choose a threading $Th(j)$ of $s_{H_j}(v_{j,2})$. Let $g_j$ be an edge-labeling of $H_j$ given by 
\[g_j(e) = \begin{cases} 2 & \text{if $e \in Th(j) - \{v_{j,2}'v_{j,2}''\}$}\\ 2 & \text{if $e = v_{j,2}v_{j,2}'$ or $e=v_{j,2}v_{j,2}''$ and $v_{j,2}'v_{j,2}'' \in Th(j)$}\\ 0 & \text{otherwise}\end{cases}.\]

Now extend $\phi$ to the edges in $E(H_j) \bigcup \{b_{j,2}, ..., b_{j,k}\}$ as follows:
\[\phi(e) = \begin{cases} \big(f_j(v_{j,i}'v_{j,i}''), g_j(e)\big) & \text{if $e=v_{j,i}v_{j,i}'$ or $e=v_{j,i}v_{j,i}''$, $1 \le i \le k$}\\ \big(f_j(e), g_j(e)\big) & \text{if $e \in s(H_j)$} \\  (2,0) & \text{if $e = b_{j,2}$} \\(0,2) & \text{if $e = b_{j,i}$ and $f_j(v_{j,i}'v_{j,i}'')=2$, $3 \le i \le k$}\\ 
(2,2) & \text{if $e = b_{j,i}$ and $f_j(v_{j,i}'v_{j,i}'') = 1$, $3 \le i \le k$ } \\

\end{cases}.\]
\item If $b_{j,1}$ is assigned label $(2,0)$ or $(0,2)$:

Since $s(H_j)$ is a $2$-connected cubic graph, there exists a $2$-factor $M_j$ containing the edge $v_{j,1}'v_{j,1}''$. Let $f_j$ be an edge-labeling of $s(H_j)$ where each edge of $M_j$ is given label $1$ and all edges of $E\big(s(H_j)\big)- M_j$ are given label $2$. Necessarily, $f_j$ is a zero-sum $\mathbb{Z}_4$-magic labeling of $s(H_j)$. 

Observing that $s_{H_j}(v_{j,1})$  has an even number of vertices of degree 2,  choose a threading $Th(j)$ of $s_{H_j}(v_{j,1})$. Let $g_j$ be an edge-labeling of $H_j$ given by 
\[g_j(e) = \begin{cases} 2 & \text{if $e\in Th(j)- \{v_{j,1}'v_{j,1}''\}$}\\ 2 & \text{if $e=v_{j,1}v_{j,1}'$ or $e=v_{j,1}v_{j,1}''$ and $v_{j,1}'v_{j,1}'' \in Th(j)$}\\ 0 & \text{otherwise}\end{cases}.\]

If $b_{j,1}$ is assigned the label $(2,0)$, we extend $\phi$ to the edges in $E(H_j) \bigcup \{b_{j,2}, ..., b_{j,k}\}$ as follows:
\[\phi(e) = \begin{cases} (f_j(v_{j,i}'v_{j,i}''), g_j(e)) & \text{if $e=v_{j,i}v_{j,i}'$ or $v=v_{j,i}v_{j,i}''$, $1 \le i \le k$}\\ (f_J(e), g_j(e)) & \text{if $e \in s(H_j)$}\\
(2,2) & \text{if $e = b_{j,i}$ and $f_j(v_{j,i}'v_{j,i}'') = 1$,$2 \le i \le k$ }\\
(0,2) & \text{if $e = b_{j,i}$ and $f_j(v_{j,i}'v_{j,i}'') = 2$, $2 \le i \le k$ }\end{cases}\]

If $b_{j,1}$ is assigned the label $(0,2)$,  we extend $\phi$ to the edges in $E(H_j) \bigcup \{b_{j,2}, ..., b_{j,k}\}$ as follows:
\[\phi(e) = \begin{cases} \big(g_j(e), f_j(v_{j,i}'v_{j,i}'')\big) & \text{if $e=v_{j,i}v_{j,i}'$ or $e=v_{j,i}v_{j,i}''$, $1 \le i \le k$}\\ \big(g_j(e), f_j(e)\big) & \text{if $e \in s(H_j)$}\\
(2,2) & \text{if $e = b_{j,i}$ and $f_j(v_{j,i}'v_{j,i}'') = 1$,$2 \le i \le k$ }\\
(2,0) & \text{if $e = b_{j,i}$ and $f_j(v_{j,i}'v_{j,i}'') = 2$, $2 \le i \le k$}\end{cases}\]
\end{itemize}
\end{enumerate}
\end{enumerate}
\end{proof}
\section{On the zero-sum $\mathbb{A}$-magicness of cubic graphs.}

Noting that every finite non-trivial  abelian group  can be expressed as the direct product of powers of finite cyclic groups $\mathbb{Z}_j$, we close this paper with a consideration of the scope of such groups $\mathbb{A}$ for which a given  cubic graph $G$ is zero-sum $\mathbb{A}$-magic. It will suffice to assume that $G$ is connected. 

Since no connected cubic graph $G$ is zero-sum $\mathbb{Z}_2$-magic, and since $G$ is zero-sum $\mathbb{A}$-magic if the direct product representation of $G$ contains a factor of $\mathbb{Z}_j$ such that $G$ is zero-sum $\mathbb{Z}_j$-magic, we have the following:

a) If $G$ has a 1-factor and chromatic index 3, then $G$ is zero-sum $\mathbb{A}$-magic for precisely all finite non-trivial abelian groups $\mathbb{A}$ except $\mathbb{Z}_2$. (See Theorems \ref{AAA} and \ref{BBB}.)

b) If $G$ is 2-edge-connected with a  1-factor and  chromatic index 4, then $G$ is zero-sum $\mathbb{A}$-magic for precisely all finite non-trivial abelian groups $\mathbb{A}$ except $\mathbb{Z}_2$ and $\mathbb{Z}_2^2$. (See Theorems \ref{AAA}, \ref{JJJ}, and \ref{BBB}.)

c) If $G$ has a 1-factor and at least one bridge, then $G$ is zero-sum $\mathbb{A}$-magic for precisely all finite non-trivial abelian groups $\mathbb{A}$ except $\mathbb{A}$ isomorphic to $\mathbb{Z}_2^k$ for  $k \geq 1$. (See  Theorems \ref{AAA} and \ref{XXX}.)

d)  If $G$ has no 1-factor (and hence at least three  bridges) and  is not zero-sum $\mathbb{A}$-magic, then  $\mathbb{A}$ must be of the form  $\mathbb{Z}_4$, $\mathbb{Z}_2^k$ for  $k \geq 1$, or  $\mathbb{Z}_2^k \times \mathbb{Z}_4$ for $k$ in some subset of $\mathbb{N}$.  (See Theorems \ref{AAA}, \ref{XXX}, and \ref{PPP}.) Without further information about $G$, it is not clear for which $k$ $G$ is (or is not) zero-sum $\mathbb{Z}_2^k \times \mathbb{Z}_4$-magic.  The lemma below indicates that either $G$ is zero-sum $\mathbb{Z}_2^k \times \mathbb{Z}_4$-magic for no positive $k$ or $G$ is  zero-sum $\mathbb{Z}_2^k \times \mathbb{Z}_4$-magic for each $k \geq c$ where $c$ is some fixed integer at most 3.

\begin{lemma} \label{LEMMA}If $G$ is a connected cubic  zero-sum $\mathbb{Z}_2^{k_0} \times \mathbb{Z}_4$-magic graph for some fixed $k_0 \geq 1$, then $G$ is zero-sum  $\mathbb{Z}_2^k \times  \mathbb{Z}_4$-magic for all $k \geq min\{3, k_0\}$. 
\end{lemma}

\begin{proof}  Since the claim is clearly true if $k_0 \leq 3$, we will  assume $k_0 \geq 4$.

If $G$ has no bridge, then $G$ is 2-edge-connected. Therefore $G$ is  zero-sum $\mathbb{Z}_2^3$-magic, and hence zero-sum $\mathbb{Z}_2^3 \times \mathbb{Z}_4$-magic. It thus suffices to assume that $G$ has at least one bridge.

Let   $\phi$ be  a   zero-sum $\mathbb{Z}_2^{k_0} \times \mathbb{Z}_4$-magic labeling of $G$, where the last coordinate  of each label in the image of $\phi$ represents the factor $\mathbb{Z}_4$.  We observe that each component $H_i$ of $G - B(G)$ is non-trivial,  since by Theorem  \ref{YYY}, the label of each bridge under $\phi$ must be the $(k_0 + 1)$-tuple $(0,0,0, ..., 0,2)$.  We further observe that $H_i$ is 2-edge-connected  and hence admits a zero-sum $\mathbb{Z}_2^3$-magic labeling $\phi_i$ by Theorem \ref{JJJ}. It is now easy to see that  $\phi'$ is a zero-sum $\mathbb{Z}_2^3 \times \mathbb{Z}_4$-magic labeling of $G$, where $\phi'$ is constructed from  $\phi$ and $\phi_i$ as follows:  For each edge $e$ in $E(G)$, let $\phi'(e) = (a,b,c,d)$ where

$d$ is the last coordinate of $\phi(e)$;

$a,b,c$ are each 0 if $e$ is a bridge of $G$;

$(a,b,c)$ respectively agree with the first 3 coordinates of $\phi_i(e)$ if $e \in E(H_i)$.
\end{proof}

\begin{theorem} \label{WWW} If $G$ is a connected cubic graph  such that $G - B(G)$ has a component that is either trivial or bipartite with an odd number of vertices of degree 2, then for every $k \geq 1$, G is not zero-sum $\mathbb{Z}_2^k \times \mathbb{Z}_4$-magic.
\end{theorem}

\begin{proof}  Suppose to the contrary that $k_0$ is a positive integer such that $G$ is zero-sum $\mathbb{Z}_2^{k_0} \times \mathbb{Z}_4$-magic. Then by the preceding lemma, there exists a   zero-sum $\mathbb{Z}_2^3 \times \mathbb{Z}_4$-magic labeling $\phi$ of $G$, where the last coordinate  of each label in the image of $\phi$ represents the factor $\mathbb{Z}_4$.

Under $\phi$, every bridge receives a label of $(0,0,0,2)$ by Theorem  \ref{YYY}, implying that $G-B(G)$ has no component isomorphic to $K_1$. Thus suppose that some component $H$ of $G-B(G)$ is bipartite with parts $X$ and $Y$  such that $X$ has an even number of vertices of degree 2 in $H$ and $Y$ has an odd number of vertices of degree 2 in $H$. Since each vertex in $X$ of degree 2 is incident to one bridge of $G$, the sum $s_X$ of the $4^{th}$ coordinates of the labels of the edges of $H$ that are incident to $X$ must be 0 mod 4. Similarly, since $Y$ has an odd number of vertices of degree 2 in $H$, the sum $s_Y$ of the $4^{th}$ coordinates  of the labels of the edges of $H$ that are incident to $Y$ must be 2 mod 4. But the bipartite assumption implies $s_X = s_Y$, giving our contradiction. 
\end{proof}

We observe that the graph $G_2$ in Figure 4 has no 1-factor and satisfies the hypotheses of Theorem \ref{WWW} by virtue of the existence of a trivial component.  Thus, by Theorems  \ref{AAA}, \ref{XXX},  and \ref{PPP}, $G_2$ is zero-sum $\mathbb{A}$-magic for precisely every finite non-trivial abelian group $\mathbb{A}$ except $\mathbb{Z}_4, \mathbb{Z}_2^k$ for $k \geq 1$, and $\mathbb{Z}_2^k \times \mathbb{Z}_4$ for $k \geq 1$.  Similarly, the graph $G_3$ in Figure 5 below  has no 1-factor and   satisfies the hypotheses of Theorem \ref{WWW} since $G_3 - B(G_3)$  has  a bipartite component isomorphic to $K_{2,3}$ with an odd number of vertices of degree 2. So it, too, is zero-sum $\mathbb{A}$-magic for precisely every finite non-trivial abelian group $\mathbb{A}$ except $\mathbb{Z}_4, \mathbb{Z}_2^k$ for $k \geq 1$, and $\mathbb{Z}_2^k \times \mathbb{Z}_4$ for $k \geq 1$.

\begin{figure}[h]
\begin{center}
\begin{tikzpicture}[]
\tikzstyle{vertex}=[circle, draw, inner sep=0pt, minimum size=6pt]
\tikzset{vertexStyle/.append style={rectangle}}
	\vertex (1) at (0,0) [fill, scale=.75] {};
	\vertex (2) at (1,0) [fill, scale=.75] {};
	\vertex (3) at (4,0) [fill, scale=.75]{};
	\vertex (4) at (2,-.94) [fill, scale=.75]{};
	\vertex (5) at (2, -2.2) [fill, scale=.75]{};
	\vertex (6) at (1.4, -3) [fill, scale=.75]{};
	\vertex (7) at (2.6, -3) [fill, scale=.75]{};
	\vertex (8) at (.5, 1.2) [fill, scale=.75]{};
	\vertex (9) at (.8, 2.1) [fill, scale=.75]{};
	\vertex (10) at (-.35, 1.55) [fill, scale=.75]{};
	\vertex (11) at (3, .75) [fill, scale=.75]{};
	\vertex (12) at (3.4, 1.9) [fill, scale=.75]{};
	\vertex (13) at (3.1, 2.75) [fill, scale=.75]{};
	\vertex (14) at (4.25, 2.35) [fill, scale=.75]{};
	\path
		(1) edge (3)
		(1) edge[bend left=50] (3)
		(1) edge[bend right=50] (3)
		(4) edge (5)
		(5) edge (6)
		(6) edge[bend right=35] (7)
		(6) edge (7)
		(5) edge (7)
		(2) edge (8)
		(8) edge (9) 
		(8) edge (10)
		(9) edge (10)
		(9) edge [bend right=35] (10)
		(11) edge (12)
		(12) edge (13)
		(12) edge (14)
		(13) edge (14)
		(13) edge[bend left=35] (14)

	;
\end{tikzpicture}
\end{center}
\caption{The graph $G_3$}
\end{figure}
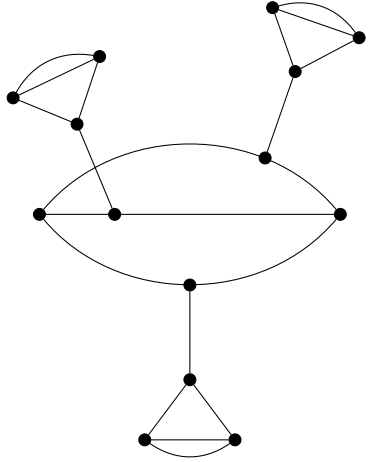

\begin{theorem} \label{ZZZ} If $G$ is a connected cubic graph such that each component of $G - B(G)$ has either precisely one vertex of degree 2 or an even number of vertices of degree 2, then $G$ is zero-sum $\mathbb{Z}_2^k \times \mathbb{Z}_4$-magic for $k \geq 3$.
\end{theorem} 

\begin{proof}  We construct a zero-sum $\mathbb{Z}_2^3 \times \mathbb{Z}_4$-magic labeling $\phi$ of $G$ whose labels represent the  factor $\mathbb{Z}_4$ in the $4^{th}$ coordinate.

Let the components of $G-B(G)$ be denoted $H_0, H_1, H_2, ..., H_q, H_{q+1}, H_{q+2}, ...., H_b$, where  those components having precisely one vertex of degree 2 are $H_j$, $0 \leq j \leq q$. Because each component $H_i$ is non-trivial and 2-edge-connected, we may find a zero-sum $\mathbb{Z}_2^3$-magic labeling $\phi_i$ of $H_i$ by Theorem \ref{JJJ}.

We establish the first three coordinates of the labels under $\phi$ as follows: for each edge $e$ in $E(G)$,   the first three coordinates of $\phi(e)$ shall respectively agree with the first three coordinates of $\phi_i(e)$ if $e \in E(H_i)$; otherwise, if $e$ is a bridge of $G$,  the first three coordinates of $\phi(e)$ shall each be 0.

To establish the 4$^{th}$ coordinate of each label under $\phi$,  we separately consider bridges, edges in $H_i$ for $0 \leq i \leq q$, and edges in $H_i$, $q+1 \leq i \leq b$.

For each bridge $e$, $\phi(e)$ shall be 2 in the $4^{th}$ coordinate.

For each  $i$,  $0 \leq i \leq q$, we let $h_i$ denote the unique vertex of degree 2 in $H_i$.  Since each such  $H_i$  is non-bipartite and thus contains an odd cycle $C_i$, we may find a shortest path $P_i$ from $h_i$ to $C_i$ if $h_i$ is not incident to $C_i$.  For each edge $e \in E(H_i)$ we define the $4^{th}$ coordinate of $\phi(e)$ to be

\vskip5pt
2 if $e$ is along $P_i$;

 the appropriate label 1 or 3 if $e$ is along  $C_i$;

0 otherwise.
\vskip5pt

For each $i$, $q+1 \leq i \leq b$, we apply Lemma \ref{QQQ}  to $H_i$, letting $Th(i)$ be a threading of $H_i$. For each edge $e \in E(H_i)$ we define the $4^{th}$ coordinate of $\phi(e)$ to be
\vskip 5pt
2 if $e$ is along a path in $Th(i)$;

 0 otherwise
\vskip 5pt

It is easy to check that $\phi$ is a zero-sum $\mathbb{Z}_2^3 \times \mathbb{Z}_4$-magic labeling of $G$. 
\end{proof}

In Figure 6, we present an example of a graph $G_4$ with no 1-factor that satisfies the conditions of Theorem \ref{ZZZ}. Thus $G_4$ is zero-sum $\mathbb{Z}_2^3  \times \mathbb{Z}_4$-magic.   However,  it is easily checked (by construction) that $G_4$ is zero-sum $\mathbb{Z}_2 \times \mathbb{Z}_4$-magic as well.  So by Theorems \ref{AAA}, \ref{XXX}, and \ref{PPP},  $G_4$ is zero-sum $ \mathbb{A}$-magic for precisely all abelian groups $\mathbb{A}$ except $\mathbb{Z}_4$ and $\mathbb{Z}_2^k$ for $k \geq 1$.

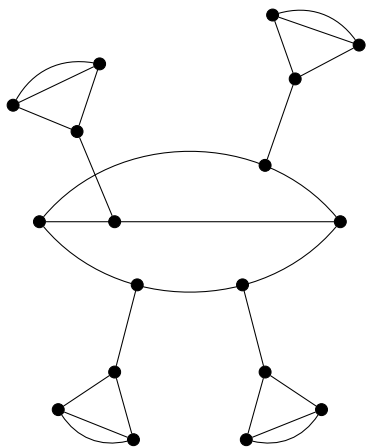
\begin{figure}[h]
\begin{center}
\begin{tikzpicture}[]
\tikzstyle{vertex}=[circle, draw, inner sep=0pt, minimum size=6pt]
\tikzset{vertexStyle/.append style={rectangle}}
	\vertex (1) at (0,0) [fill, scale=.75] {};
	\vertex (2) at (1,0) [fill, scale=.75] {};
	\vertex (3) at (4,0) [fill, scale=.75]{};
	
	\vertex (4) at (1.3,-.84) [fill, scale=.75]{};
	\vertex (5) at (2.7, -.84) [fill, scale=.75]{};
	\vertex (6) at (1, -2) [fill, scale=.75]{};
	\vertex (7) at (.25, -2.5) [fill, scale=.75]{};
	\vertex (15) at (1.25, -2.9) [fill, scale=.75]{};
	\vertex (16) at (3, -2) [fill, scale=.75]{};
	\vertex (17) at (3.75, -2.5) [fill, scale=.75]{};
	\vertex (18) at (2.75, -2.9) [fill, scale=.75]{};
	
	\vertex (8) at (.5, 1.2) [fill, scale=.75]{};
	\vertex (9) at (.8, 2.1) [fill, scale=.75]{};
	\vertex (10) at (-.35, 1.55) [fill, scale=.75]{};
	\vertex (11) at (3, .75) [fill, scale=.75]{};
	\vertex (12) at (3.4, 1.9) [fill, scale=.75]{};
	\vertex (13) at (3.1, 2.75) [fill, scale=.75]{};
	\vertex (14) at (4.25, 2.35) [fill, scale=.75]{};
	\path
		(1) edge (3)
		(1) edge[bend left=50] (3)
		(1) edge[bend right=50] (3)
		(4) edge (6)
		(6) edge (7)
		(6) edge (15)
		(7) edge (15)
		(17) edge (18)
		(17) edge[bend left=35] (18)
		(17) edge (16)
		(18) edge (16)
		(7) edge[bend right=35] (15)
		(5) edge (16)
		(2) edge (8)
		(8) edge (9) 
		(8) edge (10)
		(9) edge (10)
		(9) edge [bend right=35] (10)
		(11) edge (12)
		(12) edge (13)
		(12) edge (14)
		(13) edge (14)
		(13) edge[bend left=35] (14)

	;
\end{tikzpicture}
\end{center}
\caption{The graph $G_4$.}
\end{figure}

Obviously there are many cubic graphs that satisfy the hypotheses  of  neither of the two preceding theorems. Several questions arise. 

Does there exist a connected cubic graph $G$ that is  zero-sum $\mathbb{Z}_2^k \times \mathbb{Z}_4$-magic for some $k$ such that each component of $G - B(G)$ has an odd number of vertices of degree 2? Observing that each component of $G-B(G)$ is necessarily not bipartite (Theorem \ref{WWW}),   we note that  the graph  in Figure 2  has a 1-factor, and hence is zero-sum $\mathbb{Z}_4$-magic by Theorem \ref{AAA}. Thus, it is zero-sum $\mathbb{Z}_2^k \times \mathbb{Z}_4$-magic for all $k \geq 1$.

\noindent Accordingly, we next ask if there exists a cubic graph $G$ with no 1-factor  such that $G$ is zero-sum $\mathbb{Z}_2^k \times \mathbb{Z}_4$-magic for some $k$ and each component of $G - B(G)$ has an odd number of vertices of degree 2. We present an affirmative response in Figure 7, whose graph $G_5$ is  zero-sum $\mathbb{Z}_2 \times \mathbb{Z}_4$-magic by construction, and hence  zero-sum $\mathbb{Z}_2^k \times \mathbb{Z}_4$-magic for $k \geq 1$. %(Selected edge labels are sufficient to determine a zero-sum $\mathbb{Z}_2 \times \mathbb{Z}_4$-magic labeling.) 

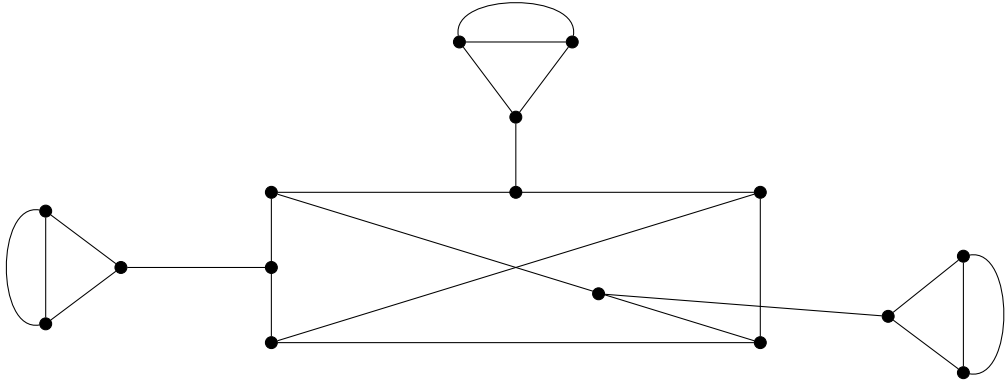
\begin{figure}[h]
\begin{center}
\begin{tikzpicture}[]
\tikzstyle{vertex}=[circle, draw, inner sep=0pt, minimum size=6pt]
\tikzset{vertexStyle/.append style={rectangle}}
	\vertex (1) at (-2,0) [fill, scale=.75] {};
	\vertex (2) at (4.5,0) [fill, scale=.75] {};
	\vertex (3) at (-2,2) [fill, scale=.75]{};
	\vertex (4) at (4.5,2) [fill, scale=.75]{};
	\vertex (5) at (1.25, 2) [fill, scale=.75]{};
	\vertex (6) at (1.25, 3) [fill, scale=.75]{};
	\vertex (7) at (.5, 4) [fill, scale=.75]{};
	\vertex (8) at (2, 4) [fill, scale=.75]{};
	\vertex (9) at (-2, 1) [fill, scale=.75]{};
	\vertex (10) at (-4, 1) [fill, scale=.75]{};
	\vertex (11) at (-5, 1.75) [fill, scale=.75]{};
	\vertex (12) at (-5, .25) [fill, scale=.75]{};
	\vertex (13) at (2.35, .65) [fill, scale=.75]{};
	\vertex (14) at (6.2, .35) [fill, scale=.75]{};
	\vertex (15) at (7.2, 1.15) [fill, scale=.75]{};
	\vertex (16) at (7.2, -.4) [fill, scale=.75]{};
	\path
		(1) edge (2)
		(1) edge (3)
		(2) edge (3)
		(3) edge (4)
		(2) edge (4)
		(1) edge (4)
		(5) edge (6)
		(6) edge (7)
		(6) edge (8)
		(7) edge (8) (8)
                 (7) edge[bend left=100] (8)
		(9) edge (10)
		(10) edge (11)
		(10) edge (12)
		(11) edge (12)
		(11) edge[bend right=100] (12)
		(13) edge (14) 
		(14) edge (15)
		(14) edge (16)
		(15) edge (16)
		(15) edge[bend left=100] (16);

%\draw (5) to node[right] {(0,2)} (6);
%\draw (6) to node[left] {(0,1)} (7);
%\draw (6) to node[right] {(0,1)} (8);
%\draw (7) to node[below] {(0,1)} (8);
%\draw (7) to node[above] {(0,2)} (8);
%\draw (5) to node[above] {(1,0)} (3);
%\draw (5) to node[above] {(1,2)} (4);
%\draw (3) to node[left] {(0,1)} (9);
%\draw (9) to node[left] {(0,1)} (1);
%\draw (10) to node[above] {(0,2)} (9);
%\draw (1) to node[below] {(1,3)} (2);
%\draw (2) to node[right] {(0,2)} (4);

	;
\end{tikzpicture}
\end{center}
\caption{The graph $G_5$.}
\end{figure}

Thus far, we have presented no graph that is zero-sum $\mathbb{Z}_2^3 \times \mathbb{Z}_4$-magic but not zero-sum $\mathbb{Z}_2 \times \mathbb{Z}_4$-magic. So we ask if every graph $G$ that is  zero-sum $\mathbb{Z}_2^3 \times \mathbb{Z}_4$-magic is necessarily  zero-sum $\mathbb{Z}_2 \times \mathbb{Z}_4$-magic. As we shall see, the answer is no. To facilitate our discussion, we present the following theorem.

\begin{theorem} \label{TTTX} Let $G$ be a connected cubic graph such that $H$ is a component of $G - B(G)$ with the following properties:

\vskip 5pt

(a) $\Delta(H) = 3$, and

(b)  no two vertices of degree 3 in $H$ are adjacent.
\vskip 5pt

\noindent Then $G$ is not zero-sum $\mathbb{Z}_2 \times \mathbb{Z}_4$-magic if the chromatic index of $s(H)$ is 4.
\end{theorem}

\begin{proof} We show that if $G$ is zero-sum $\mathbb{Z}_2 \times \mathbb{Z}_4$-magic, then $\chi'\big(s(H)\big) = 3$.

Note that $G$ has a non-empty bridge set; otherwise, $G = H$, implying a violation of condition (2). Moreover, each vertex of degree 3 of $H$ is incident to 3 distinct vertices of degree 2 of $H$, and each vertex of degree 2 in $H$ is incident to a bridge of $G$.

Let $\phi$ be a zero-sum $\mathbb{Z}_2 \times \mathbb{Z}_4$-magic labeling of $G$ under which each bridge of $G$ is necessarily assigned $(0,2)$ by Theorem \ref{YYY}. Let $W$ denote the set of vertices of degree 3 in $H$ and let $w$ be an arbitrary vertex in $W$. Then there are precisely three distinct vertices $x_1$, $x_2,$ and $x_3$ of degree 2 in $H$ that are adjacent to $w$ in $H$, and for $1 \leq i \leq 3, 1 \leq j \leq 2$, we may let $e_{i,j}$ be the edges in $H$ to which $x_i$ is incident, where $e_{i,1}$ is incident to $w$.

Since the weight of $w$ under $\phi$ is 0, we observe that the number of distinct $i$ such that $\phi(e_{i,1})$ is 1 in the first coordinate is even, and the number of distinct $i$ such that $\phi(e_{i,1})$ is 1 or 3 in the second coordinate is even; either 0 or 2. Suppose the latter is 0. Then all three edges $e_{i,1}$ are assigned either 0 or 2 in the second coordinate by $\phi$. If the second coordinate of $\phi(e_{i,1})$ is 0, then the first coordinate must be 1. And if the second coordinate of $\phi(e_{i,1})$ is 2, then (since the bridge incident to $x_i$ must be labeled $(0,2)$) $\phi(e_{i,2})$ must have 0 in the second coordinate and therefore 1 in the first coordinate, implying that $\phi(e_{i,1})$ is 1 in the first coordinate. Hence, if the number of distinct $i$ such that $\phi(e_{i,1})$ is 1 or 3 in the second coordinate is 0, we have the contradiction that $\phi(e_{i,1})$ is 1 in the first coordinate for each $i, 1 \leq i \leq 3$. Thus, the number of distinct $i$ such that $\phi(e_{i,1})$ is 1 or 3 in the second coordinate is 2, or alternatively, the number of distinct $i$ such that $\phi(e_{i,1})$ is even in the second coordinate is 1. Supposing that $e_{i_0,1}$ is the particular edge such that $\phi(e_{i_0,1})$
is even in its second coordinate, we have, by the above argument, that $\phi(e_{i_0,1})$ is 1 in its first coordinate. This implies that precisely two edges incident to $w$ have labels under $\phi$ with 1 in the first coordinate.

Observing that, in $H$, every vertex with degree 2 is incident to two edges whose second coordinates have necessarily equal parity, we let $Q$ be the subgraph of $H$ induced by the edges of $H$ that, under $\phi$, receive a label with an even second coordinate. It can be easily checked that $Q$ consists of ${\vert W \vert \over 2}$ components, each of which is a path whose terminal vertices are in $W$. Thus $Q$ induces a 1-factor $Q_{s(H)}$ in $s(H)$.

Observing that, in $H$, every vertex with degree 2 is incident to two edges whose first coordinates have necessarily equal parity, we let $R$ be the subgraph of $H$ induced by the edges of $H$ that, under $\phi$, receive a label with 1 in the first coordinate. Then $R$ is a 2-regular subgraph of $H$, and as well, $R$ has $Q$ as a subgraph. (Otherwise, $\phi$ assigns $(0,0)$ to some edge $e$.) Moreover, $R$ induces a 2-factor $R_{s(H)}$ in $s(H)$ which contains $Q_{s(H)}$.

We note that the following are pairwise disjoint: $E(Q_{s(H)})$, $E(R_{s(H)}) - E(Q_{s(H)})$ and $E\big(s(H)\big) - E(R_{s(H)})$. Since these induce a 3-edge coloring of $s(H)$, $\chi'\big(s(H)\big) = 3$. \end{proof}

Let ${\cal M}_1$ denote the infinite collection of graphs $M_1(G)$ that are formed by  the following construction:  for  2-edge-connected cubic graph $G$, let  $G_1$ be the 1-subdivision  of $G.$ Then  $M_1(G)$ is  the graph that results by  identifying each subdividing vertex  with the  vertex of degree 1 of the  martini glass graph.  We note that $M_1(G)$ is a cubic graph having order  $7\vert V(G) \vert$, size $7\vert E(G)\vert$,  $\vert E(G) \vert$ bridges, and no 1-factor.    Thus  $M_1(G)$ is neither     zero-sum $\mathbb{Z}_4$-magic nor zero-sum $\mathbb{Z}_2^k$-magic for $k \geq 1$. We note as well that $G_1$ is a bipartite component of $M_1(G) - B\big(M_1(G)\big)$ with $\vert E(G) \vert$ vertices of degree 2.   Moreover, if $\vert V(G) \vert \equiv 2$ mod $4$,  then
$\vert E(G) \vert$ is odd, implying (by Theorem \ref{WWW}) that  for each $k \geq 1$, $M_1(G)$ is not zero-sum $\mathbb{Z}_2^k \times \mathbb{Z}_4$-magic.  Theorem \ref{PPP} then implies that   if $\vert V(G) \vert \equiv 2$ mod $4$,  $M_1(G)$ is zero-sum $\mathbb{A}$-magic for every finite non-trivial abelian group $\mathbb{A}$ except  $\mathbb{Z}_4,$ $\mathbb{Z}_2^k$ for $k \geq 1$, and   $\mathbb{Z}_2^k \times \mathbb{Z}_4$ for $k \geq 1$.   On the other hand, if $\vert V(G) \vert \equiv 0$ mod 4, it follows that $\vert E(G)\vert$ is even. Thus by Theorem \ref{ZZZ}, $M_1(G)$ is zero-sum $\mathbb{Z}_2^3 \times \mathbb{Z}_4$-magic, leaving only the 
zero-sum $\mathbb{Z}_2^2 \times \mathbb{Z}_4$-magicness and zero-sum $\mathbb{Z}_2 \times \mathbb{Z}_4$-magicness of $M_1(G)$ open to question. 

We are now ready to provide an example of a cubic graph that is  zero-sum $\mathbb{Z}_2^3 \times \mathbb{Z}_4$-magic but  not zero-sum $\mathbb{Z}_2 \times \mathbb{Z}_4$-magic. Let $Pe^*$ denote the graph displayed in Figure 8. By Theorem \ref{ZZZ}, $M_1(Pe^*) $ is  zero-sum $\mathbb{Z}_2^3 \times \mathbb{Z}_4$-magic. But $M_1(Pe^*) - B\big(M_1(Pe^*)\big)$ has a 2-edge-connected component $H$ such that $s(H)$ is $Pe^*$, which has   chromatic index 4. So by Theorem \ref{TTTX}, $M_1(Pe^*)$ is not zero-sum  $\mathbb{Z}_2 \times \mathbb{Z}_4$-magic.  We leave it to the reader to establish that $M_1(Pe^*)$ is zero-sum $\mathbb{Z}_2^2 \times \mathbb{Z}_4$-magic. Noting that the order of $M_1(Pe^*)$ is 84, we believe that $M_1(Pe^*)$ is the graph of smallest order in ${\cal M}_1$ that is zero-sum $\mathbb{Z}_2^3 \times \mathbb{Z}_4$-magic but not zero-sum $\mathbb{Z}_2 \times \mathbb{Z}_4$-magic. 
Furthermore, by using a construction similar to that used for $M_1(Pe^*)$, it can be established that there are infinitely many members of ${\cal M}_1$ that are zero-sum $\mathbb{Z}_2^3 \times \mathbb{Z}_4$-magic but not zero-sum $\mathbb{Z}_2 \times \mathbb{Z}_4$-magic.

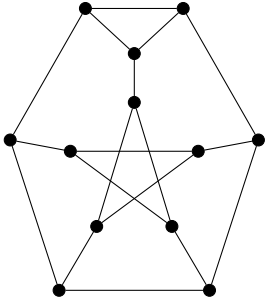
\begin{figure}[h]
\begin{center}
\begin{tikzpicture}[]
\tikzstyle{vertex}=[circle, draw, inner sep=0pt, minimum size=6pt]
\tikzset{vertexStyle/.append style={rectangle}}
	\vertex (1) at (0,0) [fill, scale=.75] {};
	\vertex (2) at (2,0) [fill, scale=.75] {};
	\vertex (3) at (-.65,2) [fill, scale=.75]{};
	\vertex (4) at (2.65,2) [fill, scale=.75]{};
	\vertex (5) at (1, 3.15) [fill, scale=.75]{};
	\vertex (6) at (1, 2.5) [fill, scale=.75]{};
	\vertex (7) at (.15,1.85) [fill, scale=.75]{};
	\vertex (8) at (1.85, 1.85)[fill, scale=.75]{};
	\vertex (9) at (.5, .85) [fill, scale=.75]{};
	\vertex (10) at (1.5, .85) [fill, scale=.75]{};
	\vertex (11) at (.35, 3.75) [fill, scale=.75]{};
	\vertex (12) at (1.65, 3.75) [fill, scale=.75]{};
	\path
		(1) edge (2)
		(2) edge (4)
		(1) edge (3)
		(5) edge (6)
		(6) edge (9)
		(6) edge (10)
		(8) edge (9)
		(8) edge (7)
		(7) edge (10)
		(3) edge (7)
		(1) edge (9)
		(2) edge (10)
		(4) edge (8)
		(11) edge (5)
		(11) edge (3)
		(11) edge (12)
		(12) edge (4)
		(12) edge (5)

	;
\end{tikzpicture}
\end{center}
\caption{The graph $Pe^*$}
\end{figure}

 Now let ${\cal M}_2$ denote the infinite collection of graphs $M_2(G)$ that are formed  by the following construction:  for  2-edge-connected cubic graph $G$, let  $G_2$ be the 2-subdivision of $G$. Then  $M_2(G)$ is   the graph that results by  identifying each subdividing vertex with the  vertex of degree 1 of the martini glass graph. We note that  $M_2(G)$ is a cubic graph with  order $13\vert V(G) \vert$, $2\vert E(G) \vert$ bridges, and no 1-factor.  Hence $M_2(G)$ is neither   zero-sum $\mathbb{Z}_4$-magic nor zero-sum $\mathbb{Z}_2^k$-magic for $k \geq 1$. We also observe that $G_2$,  a component of $M_2(G) - B\big(M_2(G)\big)$,  is bipartite if and only if $G$ is bipartite.  Since $G_2$ has an even number of vertices of degree 2 (numbering $2\vert E(G) \vert$), then by Theorem \ref{ZZZ}, $M_2(G)$ is zero-sum $\mathbb{Z}_2^3 \times \mathbb{Z}_4$-magic. The following theorem addresses the zero-sum $\mathbb{Z}_2 \times \mathbb{Z}_4$-magicness of $M_2(G)$. 

\begin{theorem} \label{FFF}
Let $G$ be a 2-edge-connected cubic graph.  Then $M_2(G)$ is zero-sum $\mathbb{Z}_2 \times \mathbb{Z}_4$-magic if and only if $\chi'(G) = 3$. 
\end{theorem}

\begin{proof}  By Theorem \ref{TTTX}, it suffices to show that if $\chi'(G) = 3$, then $M_2(G)$ is zero-sum $\mathbb{Z}_2 \times \mathbb{Z}_4$-magic.

Let $\{A, B, C\}$ be a partition of $E(G)$ whose elements represent color classes under some edge-coloring of $G$. For each edge  $e$ in $E(G)$, there exists a unique path $P(e)$ on 4 vertices in $G_2$ whose interior vertices are the subdividing vertices of edge $e$. We now construct a zero-sum $\mathbb{Z}_2 \times \mathbb{Z}_4$-magic labeling $\phi$ of $M_2(G)$, where the second coordinate of each label under $\phi$ represents the factor $\mathbb{Z}_4$. 

If $b$ is a bridge of $M_2(G)$, we let $\phi(b) = (0,2)$. 

If $e$ is an edge in color class $A$, then the labels of each edge along $P(e)$ under $\phi$ shall be $(1,1)$.

If $e$ is an edge in color class $B$, then the labels of each edge along $P(e)$ under $\phi$ shall be $(0,1)$.

If $e$ is an edge in color class $C$, then the labels of the edges along $P(e)$ under $\phi$ shall be $(1,2), (1,0), (1,2)$, respectively.

Finally, the four as yet unlabeled edges in each copy of the martini glass graph may be assigned $(1,1), (1,1), (1,1)$, and $(0,2)$ under $\phi$. 
\end{proof}

By Theorem \ref{FFF}, it follows that $M_2(G)$ is not zero-sum $\mathbb{Z}_2 \times \mathbb{Z}_4$-magic if and only if $\chi'(G) = 4$. Since the Petersen graph is the smallest cubic graph with chromatic index 4, $M_2(Pe)$ is the member of ${\cal M}_2$ of smallest order (130) that is not zero-sum $\mathbb{Z}_2 \times \mathbb{Z}_4$-magic. However, it is easily checked that the Petersen graph is zero-sum $\mathbb{Z}_2^2 \times \mathbb{Z}_4$-magic. And, since there exists a cubic graph with chromatic index 4 and order $2t$, $t \geq 5$, there is a member of ${\cal M}_2$ with order $26t$ which is not zero-sum $\mathbb{Z}_2 \times \mathbb{Z}_4$-magic.

We close this section with two questions.

\noindent 1. Is $M_1(Pe^*) $ the cubic graph of smallest order that  is zero-sum $\mathbb{Z}_2^3 \times \mathbb{Z}_4$-magic but not zero-sum $\mathbb{Z}_2 \times \mathbb{Z}_4$-magic? 

\noindent 2. Is there a cubic graph that is zero-sum  $\mathbb{Z}_2^3 \times \mathbb{Z}_4$-magic but not zero-sum $\mathbb{Z}_2^2 \times \mathbb{Z}_4$-magic?

%\begin{acknowledgements}
%\end{acknowledgements}



\end{document}